\documentclass[letterpaper, 10 pt, conference]{ieeeconf}  

\IEEEoverridecommandlockouts                              

\usepackage{epsfig} 
\usepackage{mathptmx} 
\usepackage{times} 
\usepackage{amsmath} 
\usepackage{amssymb}  
\usepackage{epstopdf}
\usepackage{color}
\usepackage{cite}

\newtheorem{theorem}{Theorem}
\newtheorem{lemma}{Lemma}
\newtheorem{mydef}{Definition}
\newtheorem{alg}{Algorithm}
\newtheorem{assumption}{Assumption}
\newtheorem{proposition}{Proposition}

\title{\LARGE \bf
Hamiltonian-based  Algorithm for Relaxed Optimal Control$^{\dag}$
}

\author{Y. Wardi, M. Egerstedt, and M.U. Qureshi$^{*}$
\thanks{$^{\dag}$Research supported in part by  the NSF  under Grant Number CNS-1239225. }
\thanks{$^*$School of Electrical and Computer Engineering, Georgia Institute of Technology, Atlanta, GA 30332. Email:
ywardi@ece.gatech.edu, magnus@ece.gatech.edu, umer.qureshi@gatech.edu.}
}

\begin{document}

\maketitle
\thispagestyle{empty}
\pagestyle{empty}

\begin{abstract}
This paper concerns a first-order algorithmic technique  for a class of  optimal control problems defined on switched-mode hybrid systems.
The salient feature of the algorithm
 is that it avoids the computation of Fr\'echet or G\^ateaux derivatives
of the cost functional, which can be time consuming, but rather moves in a projected-gradient direction that is easily computable (for a class of problems)
and does not require any explicit derivatives. The algorithm is applicable to a class of problems where a pointwise
minimizer of the
Hamiltonian is computable by a simple formula, and this includes many problems that arise in theory and applications.
The natural setting for the algorithm is the space of continuous-time relaxed controls, whose special structure renders the analysis simpler than the setting
of ordinary controls.
While the space of relaxed controls has theoretical advantages, its elements are abstract entities that may not
be
amenable to computation. Therefore, a key feature of the algorithm is that it computes adequate approximations to relaxed
controls without loosing its theoretical convergence properties.
Simulation results, including cpu times, support the theoretical developments.

\end{abstract}

\section{Introduction}

Consider the following  optimal control problem where the state equation is
\begin{equation}
\dot{x}(t)=f(x(t),u(t)),
\end{equation}
$x(t)\in R^n$ is the state variable, $u(t)\in R^k$ is the input, or control variable, time $t$ is confined
to a given interval $[0,t_f]$,
$f:R^n\times R^k\rightarrow R^n$ is the dynamic-response   function, the initial
state $x(0):=x_0\in R^n$ is given,
and  the cost functional
is
\begin{equation}
J=\int_{0}^{t_f}L(x(t),u(t))dt+\phi(x(t_f))
\end{equation}
for   cost functions $L:R^n\times R^k\rightarrow R$ and $\phi:R^n\rightarrow R$.
Let $U\subset R^k$ be a compact set, and consider the constraint that $u(t)\in U$ for every $t\in[0,t_f]$.
To ensure that Eq. (1) has a unique, continuous and piecewise-differentiable   solution,   the integral in Eq. (2) is well defined,
and other conditions mentioned in the sequel are satisfied, we make the following assumption.
\begin{assumption}
(1).  The function $f(x,u)$ is twice-continuously differentiable in $x\in R^n$ for every $u\in U$;
the functions $f(x,u)$,  $\frac{\partial f}{\partial x}(x,u)$, and
$\frac{\partial^2 f}{\partial x^2}(x,u)$ are locally-Lipschitz continuous in $(x,u)\in R^n\times U$; and
there exists $K>0$ such that, for every $x\in R^n$ and for every $u\in U$,
$||f(x,u)||\leq K(||x||+1)$.
(2). The function $L(x,u)$ is continuously differentiable in $x\in R^n$ for every
$u\in U$; and the functions $L(x,u)$ and   $\frac{\partial L}{\partial x}(x,u)$ are locally-Lipschitz continuous in
$(x,u)\in R^n\times U$.
\end{assumption}

Define an admissible control to be a function $u:[0,t_f]\rightarrow U$ which is piecewise continuously differentiable
and has a finite number of points of non-continuity. We denote an admissible control by the bar notation
$\bar{u}:=\{u(t)\}_{t\in[0,t_f]}$ to distinguish it from the value $u(t)\in U$ for a given $t\in[0,t_f]$.
Denote by ${\cal U}$ the space of admissible controls.

Observe that $J$ as defined by Eqs. (1) -(2) can be viewed as a function of $\bar{u}\in{\cal U}$, hence denoted
by $J(\bar{u})$ and called
a cost functional. The optimal control problem is to minimize $J(\bar{u})$ over $\bar{u}\in{\cal U}$.
Note that while the constraint set $U$ is assumed to be compact it need not be convex,
may have an empty interior  and even be
a finite set.

The optimal control problem can be viewed as a constrained optimization
problem on a function space, namely the space ${\cal U}$. Optimization algorithms defined on
an infinite-dimensional space (such as ${\cal U}$) must be discretized in order
to be solved by a numerical algorithm. There are basically two approaches to this: one discretizes the
problem first and then applies nonlinear-programming techniques to the resultant finite-dimensional problem,
and the other defines an algorithm in terms of the infinite-dimensional variable  and then discretizes the computations.
The latter approach was formalized in \cite{Polak97} and pursued in \cite{Caldwell11,Caldwell12,Miller13, Vasudevan13,Caldwell16,Wardi16}, and we   also adopt it for the following two reasons: (i).
The algorithm described below makes explicit use of the Hamiltonian function and the maximum principle,
hance it is more natural to describe and analyze it  in the  continuous-time, infinite-dimensional problem setting.
(ii). Discretization often involves balancing precision with complexity of computations, and discretizing the computations instead
of the problem affords the user considerable flexibility in determining the complexity level one iteration at a time.

We set the optimal control problem in the framework of relaxed controls \cite{McShane67,Young69,Warga72,Gamkrelidze78},
described in detail in the next section.
A relaxed control is a mapping $\mu$ from the interval $[0,t_f]$ into the set of Borel probability measures
on the set $U$. It is an extension of the notion of the ordinary control, defined as a Lebesgue-measurable
function $u:[0,t_f]\rightarrow U$.  An ordinary control can be viewed as
a relaxed control by associating with each $t\in[0,t_f]$ the Dirac measure at $u(t)$.

The space of relaxed controls is compact (in a suitable sense, discussed below) as well as convex, whereas the space of admissible controls typically is not compact and may not be convex. Therefore the setting of relaxed controls provides certain
theoretical advantages over
the setting of admissible controls, such as the existence of solutions to the optimal control problems   \cite{Warga72} and the simplicity of analysis of conceptual (abstract) algorithms \cite{Wardi16}. However, implementation may be more difficult due to the
fact that relaxed controls are more abstract objects than admissible controls. Much of the analysis in the sequel addresses
this point by identifying a class of hybrid systems where implementable algorithms are possible with limited complexity.
This class of systems is broad enough to include  various problems of practical and theoretical interest.

Consider the Hamiltonian function $H(x,u,p):=p^{\top}f(x,u)+L(x,u)$, where $p\in R^n$ is the costate (adjoint) variable.
Fix $x\in R^n$ and $p\in R^n$, and consider $H(x,u,p)$ as a function of $u\in U$. By Assumption 1 this function is continuous
and, since $U$ is compact, it admits a minimum at $U$. We call such a  minimum point a
{\it pointwise minimizer of the Hamiltonian}. The algorithm is suitable for  problems where a pointwise minimum
of the Hamiltonian is computable by a simple formula. \footnote{This is not the same as the statement that a control satisfying
a two-point boundary value problem can be easily computed, which we are not
making.}

Given a relaxed control, let $\{x(t)\}_{t\in[0,t_f]}$ and $\{p(t)\}_{t\in [0,t_f]}$ be the associated state trajectory and costate trajectory. The algorithm defines  a descent direction by computing a pointwise minimizer of the Hamiltonian
$H(x(t),u,p(t))$ for a finite set of points $t\in[0,t_f]$. It then takes a suitable step in that direction to compute
the next relaxed control. In an abstract setting of the algorithm it appears that its complexity grows without a bound in successive iterations, and the main contribution of the paper is to limit the complexity without detracting from the algorithm's convergence
properties.

A preliminary version of the algorithm and its analysis in the abstract setting of Eqs. (1) - (2)   have been presented in \cite{Wardi16},
but it was not applicable to a general class of switched-mode hybrid systems. The goal of this paper is to close this gap, and extend the algorithm
to systems and problems defined as follows. The state equation is
\begin{equation}
\dot{x}(t)\in\big\{f_{i}(x(t),u_i(t))~:~i=1,\ldots,M\big\},
\end{equation}
where the functions $f_{i}:R^n\times R^{k_i}\rightarrow R^n$, $i=1,\ldots,M$ represent different modes of the
 system, $u_i\in U_i\subset R^{k_i}$, and  the mode-dependent set $U_i$ is compact.  At each time $t\in[0,t_f]$ the control variable,
  denoted by $v(t)$, consists of the mode-index $i$ and the continuous-valued
control, namely, $v(t)=(i(t),u_{i(t)}(t))$, where $i(t)\in\{1,\ldots,M\}$ and $u_{i(t)}(t)\in U_{i(t)}$. Denote by $V$ the set of pairs $v=(i,u_i)$
such that $i\in\{1,\ldots,M\}$ and $u_i\in U_i$,  and let $\bar{v}$ denote an admissible control, namely the function
$\{v(t)\}_{t\in[0,t_f]}$ which is piecewise continuously differentiable and has finite numbers of discontinuities.
Denote by ${\cal V}$ the space of admissible controls.
 The cost functional
for the optimal control problem is
\begin{equation}
J=J(\bar{v})=\int_{0}^{t_f}L_{i(t)}(x(t),u_{i(t)}(t))dt+\phi(x(t_f)),
\end{equation}
where $L_i:R^{n}\times R^{k_i}\rightarrow R$, $i=1,\ldots,M$, are mode-dependent running cost functions, and $\phi:R^n\rightarrow R$ is
a final-state cost function. The optimal control problem is to minimize $J=J(\bar{v})$ over $\bar{v}\in{\cal V}$.

The rest of the paper is organized as follows.
Section II recounts relevant existing results,   Section III
extends the algorithm so as to be applicable  for a class of the aforementioned  switched-mode problems, and Section IV provides simulation results.

\section{Survey of Established Results}
This section surveys existing results which are relevant to the developments made in  the sequel. In particular we discuss E.
Polak's framework of infinite-dimensional optimization, the foundations of relaxed controls, and our preliminary algorithm presented in
\cite{Wardi16}.

\subsection{Optimality functions and sufficient descent}
Let ${\cal M}$ be a Hausdorff topological space with a Borel measure ${\cal F}$, and let $J:{\cal M}\rightarrow R$ be
a measurable function. Consider the abstract problem of minimizing $J(\bar{\mu})$ over $\bar{\mu}\in{\cal M}$.
For a given necessary optimality condition, let $\Delta\subset {\cal M}$ be the set of points $\bar{\mu}\in{\cal M}$ where it
is satisfied, and suppose that $\Delta$ is measurable. Let $\theta:{\cal M}\rightarrow R^-$ be
a measurable function. Polak defines $\theta$ to be an
 {\it optimality function} if (i) $\theta(\bar{\mu})=0$ iff $\bar{\mu}\in\Delta$, and (ii) $|\theta(\bar{\mu})|$ provides a measure
 of the
 extent to which $\bar{\mu}$ fails to satisfy the optimality condition.

Consider an iterative algorithm for minimizing $J$ over ${\cal M}$, and let $\bar{\mu}_{j}$, $j=1,2,\ldots$, be
a sequence of points it computes from a given
initial point $\bar{\mu}_{0}$. Suppose that we can represent the computation of $\bar{\mu}_{j+1}$ from $\bar{\mu}_{j}$ via the
notation $\bar{\mu}_{j+1}=T(\bar{\mu}_{j})$, for a measurable mapping
$T:{\cal M}\rightarrow{\cal M}$.
\begin{mydef}
The algorithm
is a {\it sufficient-descent} method with respect to $\theta$ if (i)
 for every $\bar{\mu}\in{\cal M}$, $J(T(\bar{\mu}))\leq J(\bar{\mu})$;
and (ii) for every $\eta>0$ there exists $\delta>0$ such that, for every $\bar{\mu}\in{\cal M}$ such that $\theta(\bar{\mu})<-\eta$,
\begin{equation}
J(T(\bar{\mu}))-J(\bar{\mu})<-\delta.
\end{equation}
\end{mydef}
In finite-dimensional optimization, where  ${\cal M}$ is a closed subset of $R^n$, it is common to characterize the convergence of an algorithm by the condition that every accumulation point of a computed sequence $\{\bar{\mu}_{j}\}_{j=1}^{\infty}$, is
contained in $\Delta$. In infinite-dimensional problems, there is no guarantee that a bounded sequence
would have an accumulation point, nor is it guaranteed that the
set $\Delta$ is non-empty.  For this reason the concept of optimality functions was defined in order to characterize convergence of the algorithm
by the following limit,
\begin{equation}
\lim_{j\rightarrow\infty}J(\bar{\mu}_{j})=0.
\end{equation}
The notion of sufficient descent guarantees this condition via the following
result.
\begin{proposition}
Suppose that $J(\bar{\mu})$ is bounded from below over
$\bar{\mu}\in{\cal M}$. If the algorithm is of sufficient descent, then every sequence $\{\bar{\mu}_{j}\}_{j=1}^{\infty}$ of iteration points computed by
the algorithm, satisfies Eq. (6).
\end{proposition}

The  proof is immediate; see \cite{Polak97}.\hfill$\Box$

We mention that these concepts form the foundations of a general framework for the development and analysis of algorithms, targeted  at
infinite-dimensional optimization problems  including optimal control.  For details, please see Chapter 4 in \cite{Polak97}.

\subsection{Relaxed Controls}
The theory of relaxed controls was developed in the late nineteen-sixties \cite{McShane67,Young69,Warga72,Gamkrelidze78}, and more recent surveys can be found in  \cite{Lou09,Vinter00,Berkovitz13}. This subsection
 summarizes its main points which are relevant to the present paper.

 Consider the optimal control problem defined in Section I.
 Let $M$ denote the space of Borel  probability measures on the set $U$, and denote by
 $\mu$ a particular measure in $M$.
A relaxed control  associated with the system (1) is a mapping $\mu:[0,t_{f}]\rightarrow M$ which is measurable
in the following sense: For every
continuous function $\zeta:U\rightarrow R$, the function $\int_{U}\zeta(u)d\mu(t)$ is Lebesgue measurable
in $t$.
We denote
the space of relaxed controls by ${\cal M}$, and   an element in this space is denoted by
$\bar{\mu}:=\{\mu(t)\}_{t\in[0,t_f]}$ to distinguish it from the points $\mu(t)\in M$ for every $t\in[0,t_f]$.

The space of relaxed controls is endowed with the weak star topology whereby
$\lim_{k\rightarrow\infty}\bar{\mu}_{k}=\bar{\mu}$ if
for every function  $\psi:[0,t_{f}]\times U\rightarrow R$ which is measurable and absolutely integrable
in $t$ on $[0,t_{f}]$ for every $u\in U$, and
continuous on $U$ for every $t\in[0,t_{f}]$,
\begin{equation}
\lim_{k\rightarrow\infty}\int_{0}^{t_{f}}\int_{U}\psi(t,u)d\mu_{k}(t)dt\ =\ \int_{0}^{t_{f}}\int_{U}\psi(t,u)d\mu(t)dt.
\end{equation}
The space ${\cal M}$ is compact in the weak star topology.

As noted earlier, every ordinary control $\bar{u}$ is associated with the relaxed control
$\bar{\mu}$ by defining $\mu(t)$ as the Dirac measure on $u(t)$. Therefore the space of ordinary controls is contained in the space
of relaxed controls, ${\cal M}$. Moreover, the space of ordinary controls is dense in ${\cal M}$ in the weak-star topology,
and since the space of admissible controls, ${\cal U}$, is dense in the space of ordinary controls in the $L^1$ topology (and hence in the weak-star topology as well), we have that ${\cal U}$ is dense in ${\cal M}$ in the weak-star topology.

An extension of the aforementioned optimal control
problem to the setting of relaxed controls is defined by generalizing Eqs. (1) and (2) to the
relaxed state equation and cost functional, defined as follows.
For a relaxed control $\bar{\mu}$,  the relaxed state equation
 is
\begin{equation}
\dot{x}(t)=\int_{U}f\big(x(t),u\big)d\mu(t)
\end{equation}
with the same boundary condition $x_{0}=x(0)$ as for (1),
and the relaxed  cost functional is
\begin{equation}
J(\bar{\mu})=\int_{0}^{t_{f}}\int_{U}L\big(x(t),u\big)d\mu(t) dt+\phi(x(t_f)).
\end{equation}
The relaxed optimal control problem is to minimize $J(\bar{\mu})$ over $\bar{\mu}\in{\cal M}$.
The relaxed costate (adjoint)
variable, denoted by $p(t)$, is defined by the equation
\begin{equation}
\dot{p}(t)=-\int_{U}\Big(\frac{\partial f}{\partial x}\big(x(t),u\big)^{\top}p(t)+\frac{\partial L}{\partial x}\big(x(t),u\big)^{\top}\Big)d\mu(t)
\end{equation}
 with the boundary condition $p(t_{f})=\nabla\phi(x(t_f))$,
and the relaxed Hamiltonian is defined as
\begin{align}
H\big(x(t),\mu(t),p(t)\big) \nonumber \\
=\int_{U}\Big(p(t)^{\top}f\big(x(t),u\big)+L\big(x(t),u\big)\Big)d\mu(t).
\end{align}
The maximum principle is in force and provides a natural necessary optimality condition for the relaxed optimal control
problem (see \cite{Vinter00}). It states that if $\bar{\mu}\in{\cal M}$ is a solution  for the relaxed optimal control problem then
$\mu(t)$ minimizes the Hamiltonian at almost every time-point $t\in[0,t_{f}]$.

In these notational usages we do not distinguish between relaxed controls and ordinary controls. For instance, if
$\bar{\mu}$ is associated with an ordinary control $\bar{u}$ in the manner described above, we write $\bar{\mu}=\bar{u}$,
and note that Eq. (8) is reduced to Eq. (1), and  similarly, the relaxed Hamiltonian in (11) is reduced to the ordinary Hamiltonian $H(x(t),u(t),p(t))$.

\subsection{Preliminary version of the algorithm}

This subsection describes the algorithm presented in \cite{Wardi16}, and recounts some theoretical results whose proofs can
be found in the latter reference. The setting of Eqs. (1) - (2) is assumed.

Consider a relaxed control  $\bar{\mu}\in{\cal M}$. Let  $x(t)$ and $p(t)$ denote the state variable and costate variable defined by Eqs. (8) and (10), respectively. For a given $\bar{\nu}\in{\cal M}$, the function
$
\Theta(\bar{\nu}):=\int_{0}^{t_f}H(x(t),\nu(t),p(t))dt
$
from ${\cal M}$ to $R$ is continuous in the weak-star topology  in ${\cal M}$
(see \cite{Lou09}),  and since ${\cal M}$ is compact in the weak
star topology, it attains a minimum in ${\cal M}$. Therefore the following function, $\theta(\bar{\mu})$ is well defined:
\begin{equation}
\theta(\bar{\mu}):=
\min_{\bar{\nu}\in{\cal M}}\Big(H(x(t),\nu(t),p(t))-H(x(t),\mu(t),p(t))\Big).
 \end{equation}
 We note that $\theta(\bar{\mu})$
 is an optimality
 function with respect to the relaxed maximum principle.

 For a given $t\in[0,t_f]$, consider the Hamiltonian function $H(x(t),u,p(t))$ as a function of
$u\in U$. By Assumption 1 this function is continuous, and since
$U$ is compact, it attains its minimum there. Let $u^{*}(t)$ be a minimum point. We have the following simple yet useful result:
\begin{lemma}
\cite{Wardi16}.  Given $t\in[0,t_f]$, for every Borel probability measure $\nu\in M$,
\begin{equation}
H(x(t),u^{*}(t),p(t))\leq \int_{U}H(x(t),\nu,p(t))d\nu(u).
\end{equation}
\end{lemma}
This result states that, for given $x(t)\in R^n$ and $p(t)\in R^n$,  the minimum of the Hamiltonian $H(x,\nu,p)$
over $\nu\in M$ is obtained by the Dirac measure at a point $u^{*}(t)\in U$.

It is tempting to conclude that
\begin{equation}
\theta(\bar{\mu})=\int_{0}^{t_f}\big(H(x(t),u^{*}(t),p(t))-H(x(t),\mu(t),p(t))\big)dt,
\end{equation}
and hence that $\theta(\bar{\mu})$ is realized by an ordinary control, i.e., $\bar{u}^{*}$. However, the function $\{u^{*}(t)\}_{t\in[0,t_f]}$
might not be measurable and hence does not qualify as an ordinary control. Thus, while $\theta(\bar{\mu})$ always can be realized by
a relaxed control, henceforth denoted by $\bar{\mu}^{*}$, it cannot necessarily be realized by an ordinary control.
However, given $\eta\in(1,0)$, it is always possible to find an admissible control
$\bar{u}^{*}_{\eta}$ such that
\begin{eqnarray}
\int_{0}^{t_f}\Big(H(x(t),u^{*}_{\eta}(t),p(t))-\\ \nonumber
H(x(t),\mu(t),p(t))\Big)dt<(1-\eta)\theta(\bar{\mu}).
\end{eqnarray}
For instance, take   $\bar{u}_{\eta}^{*}$ to be a  zero-order interpolation of points $u^{*}(t)$ for $t$ in a finite grid
${\cal G}\subset[0,t_f]$
of equally-spaced points.   Although the formulation of the algorithm
 (below) does not specify the particular choice of
  $\bar{u}^{*}_{\eta}$, this example can be practical as long as a pointwise minimizer of the Hamiltonian, $u^{*}(t)\in U$,
  can be easily computed by a simple formula for a given $t\in[0,t_f]$.

Similarly to the presentation in Section II.A, we describe the algorithm by specifying its main loop, represented by a mapping $T:{\cal M}\rightarrow{\cal M}$. Thus, starting from an initial
guess $\bar{\mu}_{0}\in{\cal M}$, the algorithm computes, iteratively,
a sequence $\{\bar{\mu}_{j}\}_{j\geq 1}$ such that, for all $j$, $\bar{\mu}_{j+1}=T(\bar{\mu}_{j})$. The mapping $T(\bar{\mu})$ is characterized by
two quantities: a direction, and a step size.
The direction we choose is an admissible  control $\bar{u}_{\eta}^{*}$ satisfying Eq. (15). For the step size we
choose the Armijo stepsize as formally specified by the algorithm below. We point out that the Armijo step size has been used extensively
in gradient-descent optimization including optimal control problems
 \cite{Polak97,Caldwell11,Caldwell12,Vasudevan13,Caldwell16}. It is described
as a part of the following formulation.

Given constants $\eta\in(0,1)$, $\alpha\in(0,1)$, and $\beta\in(0,1)$.

Given $\bar{\mu}\in{\cal M}$, compute $T(\bar{\mu})\in{\cal M}$ as follows.
\begin{alg}
{\it Step 1:} Compute $\{x(t)\}$ and $\{p(t)\}$, $t\in[0,t_f]$, by numerical means using Eqs. (8) and (10).\\
{\it Step 2:} Compute an admissible control $\bar{u}^{*}_{\eta}\in{\cal U}$ satisfying Eq. (15).\\
{\it Step 3:} Compute the largest $\lambda$ from the set $\{1,\beta,\beta^2,\ldots\}$ such that,
\begin{equation}
J\big((1-\lambda)\bar{\mu}+\lambda\bar{u}^{*}_{\eta}\big)-J(\bar{\mu})<\alpha\lambda\theta(\bar{\mu}).
\end{equation}
Denote the resulting $\lambda$ by $\lambda_{\bar{\mu}}$.\\
{\it Step 4:} Set $T(\bar{\mu})=(1-\lambda)\bar{\mu}+\lambda\bar{u}^{*}_{\eta}$.\hfill$\Box$
\end{alg}

We remark that the term $(1-\lambda)\bar{\mu}+\lambda\bar{u}^{*}_{\eta}$ indicates a convex combination of the relaxed controls $\bar{\mu}$ and $\bar{u}^{*}_{\eta}$
(the latter, too, as a relaxed control) in the sense of measures. Thus, the state trajectory of $(1-\lambda)\bar{\mu}+\lambda\bar{u}^{*}_{\eta}$, and its cost
$J((1-\lambda)\bar{\mu}+\lambda\bar{u}^{*}_{\eta})$, are defined via Eqs. (8) and (9) with the measure $(1-\lambda)\bar{\mu}+\lambda\bar{u}^{*}_{\eta}$ replacing
$\bar{\mu}$.

\begin{theorem}
For every $\eta\in(0,1)$ there exists $\bar{\alpha}\in(0,1)$ such that, for every choices of
$\alpha\in(0,\bar{\alpha})$ and  $\beta\in(0,1)$,
Algorithm 1 has the property of sufficient descent with respect to the relaxed maximum principle.\hfill$\Box$
\end{theorem}

If the pointwise minimizer of the Hamiltonian can be computed by a simple formula then the admissible control $\bar{u}^{*}_{\eta}$ can be
computed easily as well, and this may result in  Algorithm 1 being  efficient, as simulation results, presented in Section IV, will show. However, there is a difficulty that first must be overcome before the algorithm can be considered.
To explain it, consider
the case where, in Algorithm 1, the relaxed control  $\bar{\mu}$ is actually an admissible control, namely $\bar{\mu}=\bar{u}\in{\cal U}$. By Step 4 of the algorithm,
$T(\bar{\mu})$ is a convex combination (in the sense of measures) of two admissible controls,
namely $T(\bar{\mu})=(1-\lambda)\bar{\mu}+\lambda\bar{u}^{*}_{\eta}$. Likewise, the result of $j$ iterations of the algorithm (for $j\geq 1$) is a  convex combination of $j+1$ admissible
controls. These are not ordinary controls but relaxed controls, and as we can see, their convex dimensionality increases with $j$.
We will resolve this issue for a class of switched-mode systems.

\section{Switched-Mode Hybrid Systems}

In recent years there has been a mounting interest in the hybrid optimal control problem whose state equation
and cost functional are
defined by
Eqs. (3) and (4). A number of algorithmic approaches emerged, including
 first- and second-order gradient-descent techniques \cite{Xu02,Shaikh02,Xu04,Egerstedt06,Caldwell11,Vasudevan13},
zoning algorithms based on the geometric properties of
the underlying systems \cite{Shaikh03,Caines05,Shaikh05,Shaikh07},   projection-based algorithms
\cite{Caldwell11,Caldwell12,Miller13,Caldwell16},  methods based on dynamic programming and convex optimization
\cite{Hellund99}, and needle-variations techniques   \cite{Attia05,Gonzalez10,Wardi12,Taringoo13}.    A relaxed-control
algorithm was proposed in Ref.
 \cite{Ge75}. An embedded control approach
 was analyzed  in  \cite{Bengea05} and tested in conjunction with MATLB's fmincon nonlinear-programming solver \cite{Bengea05,Meyer12}.
 A comprehensive survey of algorithmic techniques for the hybrid optimal control problem can be found in
  \cite{Lin14}.

An explicit characterization of relaxed controls in the setting of hybrid systems defined by Eqs. (3)-(4) may
be complicated. In particular,  the fact that each constraint set $U_i$ depends on the mode-index $i$
can render challenging the handling of relaxed controls in a way that is amenable to efficient
computation by an algorithm. However,
Refs. \cite{Bengea05,Berkovitz13} resolve this difficulty by defining and  considering a space of  embedded controls lying
  between the space of ordinary controls and the space of
 relaxed controls.  Embedded controls are defined as follows \cite{Bengea05,Berkovitz13}: Let $W$ denote the set of $M$-tuples  of pairs, $\big((\alpha_1,u_1),(\alpha_2,u_2),\ldots,(\alpha_M,u_M)\big)$, where
$\alpha_{i}\geq 0$
$\forall i=1,\ldots,M$,
 $\sum_{i=1}^M\alpha_i=1$,
and $u_i\in U_i$, $i=1,\ldots,M$.
 An embedded control is
a Lebesgue measurable function $w:[0,t_f]\rightarrow W$, and we denote the space of embedded controls
by
${\cal W}$. Furthermore, we denote an embedded control $\{w(t)\}_{t\in[0,t_f]}$ by $\bar{w}\in{\cal W}$.

For a given $\bar{w}\in{\cal W}$, the state
 equation is defined by
\begin{equation}
\dot{x}(t)=\sum_{i=1}^{M}\alpha_i(t)f_i(x(t),u_i(t))
\end{equation}
with the given boundary condition $x(0)=x_0\in R^n$,
the cost functional has the form
\begin{equation}
J=\sum_{i=1}^{M}\int_{0}^{t_f}\alpha_{i}(t)L_i(x(t),u_i(t))dt+\phi(x(t_f)),
\end{equation}
and the costate equation is
 \begin{eqnarray}
 \dot{p}(t)=
 -\sum_{i=1}^{M}
 \alpha_{i}(t)
 \Big[\Big(\frac{\partial f_i}{\partial x}\big(x(t),u_{i}(t)
 \big)\Big)^{\top}p(t)\nonumber \\
 +\Big(\frac{\partial L_{i}}{\partial x}L_{i}\big(x(t),u_{i}(t)\big)\Big)^{\top}\Big]
 \end{eqnarray}
 with the boundary condition $p(t_f)=\nabla\phi(x(t_f))$. For a detailed expositions of embedded controls, see \cite{Bengea05,Berkovitz13}.

Ref. \cite{Bengea05} derived first- and second-order optimality conditions for various optimal control problem
formulations. In particular, for the problem defined  in this paper, the space of embedded controls is dense in the
space of relaxed controls
in the weak-star topology. If a problem includes  constraints on the final state then this density holds no more.
However, penalty functions can be used to alter the problem into one without final-state constraints.
 Our algorithm, defined below, computes in
the space of embedded controls.

In the last paragraph of Section II we described the challenge inherent in Algorithm 1 (in an abstract setting) due to
the increasing convex dimensionality  of relaxed controls computed in successive iterations.
  We address this difficulty for the hybrid optimal control problem in the following way.
Suppose that  the input $\bar{\mu}$ to a given iteration is an embedded control. The iteration computes the term $T(\bar{\mu})$
which is a relaxed control as defined by Step 4. Now we modify the algorithm by adding the computation of
an embedded control $\bar{y}\in{\cal W}$ having the property  that $J(\bar{y})\leq J(T\bar{\mu}))$. Since (by Theorem 1) Algorithm 1
has the sufficient-descent property with respect to the maximum principle, the latter inequality ensures that the sufficient-descent
property is maintained by the modified algorithm. We point out that $\bar{y}$ is not a projection of $T(\bar{\mu})$ onto the
space of embedded controls as in the projection-based algorithms \cite{Caldwell11,Caldwell12,Miller13, Caldwell16} mentioned above.

The class of problems for which the algorithm is applicable is defined as follows.

\begin{assumption}
For every $x\in R^n$, and for every $i=1,\ldots,M$,
(i) $f_i(x,u_i)$ is affine in $u_i\in U_i$,
and (ii) $L_i(x,u_i)$ is convex in $u_i\in U_i$.\hfill$\Box$
\end{assumption}

Part (i) of the assumption means that for every
$i=1,\ldots,M$
there exist functions $\Phi_i:R^n\rightarrow R^{n\times k_i}$ and $\Psi_i:R^n\rightarrow R^n$ such that,
\begin{equation}
f_i(x,u_i)=\Phi_i(x)u_i+\Psi_i(x).
\end{equation}
\begin{assumption}
For every $i=1,\ldots,M$, (i)
the functions $\Phi_{i}(x)$ and $\Psi_{i}(x)$ are twice-continuously differentiable, and (ii) the function $L_{i}(x,u)$ satisfies Assumption 1.
\end{assumption}

Given two embedded controls, $\bar{w}_{1}\in{\cal W}$ and $\bar{w}_{2}\in{\cal W}$, and given $\lambda\in[0,1]$, we use the
notation $(1-\lambda)\bar{w}_{1}\oplus\lambda\bar{w}_{2}$ to designate the convex combination $(1-\lambda)\bar{w}_{1}+\lambda\bar{w}_{2}$ in the sense of measures. Thus, if
$w_1(t)=\big((\alpha_{1,1}(t),u_{1,1}(t)),\ldots,(\alpha_{1,M}(t),u_{1,M}(t)\big)$
and $w_2(t)=\big((\alpha_{2,1}(t),u_{2,1}(t)),\ldots,(\alpha_{2,M}(t),u_{2,M}(t)\big)$, then the
state equation of $(1-\lambda)\bar{w}_{1}\oplus\lambda\bar{w}_{2}$ is
\begin{eqnarray}
\dot{x}(t)=(1-\lambda)\sum_{i=1}^{M}\alpha_{1,i}(t)f_{i}(x(t),u_{1,i}(t))\nonumber \\
+\lambda \sum_{i=1}^{M}\alpha_{2,i}(t)f_{i}(x(t),u_{2,i}(t)),
\end{eqnarray}
and the cost functional is
\begin{eqnarray}
J\big((1-\lambda)\bar{w}_{1}\oplus\lambda\bar{w}_{2}\big)\nonumber \\
=
(1-\lambda)\sum_{i=1}^{M}
\int_{0}^{t_f}\alpha_{1,i}(t)L_{i}(x(t),u_{1,i}(t))dt\nonumber \\
+\lambda\sum_{i=1}^{M}
\int_{0}^{t_f}\alpha_{2,i}(t)L_{i}(x(t),u_{2,i}(t))dt+\phi(x(t_f).
\end{eqnarray}

The following algorithm is described by specifying its main loop, as for Algorithm 1. The input to the main loop is  $\bar{w}\in{\cal W}$, and the corresponding
output is
$\bar{y}\in{\cal W}$. The first three steps of the algorithm  are identical to those of Algorithm 1, and the difference is
in Step 4.

Given constants $\eta\in(0,1)$, $\alpha\in(0,1)$, and $\beta\in(0,1)$.

\begin{alg} Given\\
 $\bar{w}=\big((\alpha_1(t),u_1(t)),\ldots,(\alpha_M(t),u_M(t))\big),~t\in[0,t_f]$.

{\it Step 1:} Compute $\{x(t)\}$ and $\{p(t)\}$, $t\in[0,t_f]$, by numerical means, using Eqs. (17) and (19). \\
{\it Step 2:} Compute an admissible control $\bar{u}^{*}_{\eta}\in{\cal U}$ satisfying Eq. (15). For every $t\in[0,t_f]$,
$u_{\eta}^{*}(t)=(j(t),u_{j(t)}^{*}(t))$ for some $j(t)\in\{1,\ldots,M\}$ and $u_{j(t)}^{*}(t)\in U_j(t)$, and we can
view it as an embedded control of the form
$u_{\eta}^{*}(t)=\big((\alpha^*_{1}(t),u_{1}^{*}(t)),\ldots,(\alpha^*_{M}(t),u_{M}^{*}(t))\big)$,
where, $\alpha^*_{j(t)}(t)=1$,
 and for all $i\neq j(t)$, $\alpha^*_{i}(t)=0$.
\\
{\it Step 3:} Compute the largest $\lambda$ from the set $\{1,\beta,\beta^2,\ldots\}$ such that,
\begin{equation}
J\big((1-\lambda)\bar{w}\oplus\lambda\bar{u}^{*}_{\eta}\big)-J(\bar{w})<\alpha\lambda\theta(\bar{w}).
\end{equation}
Denote the resulting $\lambda$ by $\lambda_{\bar{w}}$.\\
{\it Step 4:} For every $t\in[0,t_f]$, define
$\gamma_{i}(t)=(1-\lambda_{\bar{w}})\alpha_{i}(t)+\lambda_{\bar{w}}\alpha_{i}^{*}(t)$, and
define $\epsilon_{i}(t)=\lambda_{\bar{w}}\alpha_{i}^{*}(t)/\gamma_{i}(t)$. For every $i=1,\ldots,M$ define
$\tilde{u}_{i}(t)=(1-\epsilon_{i}(t))u_{i}(t)+\epsilon_{i}(t)u_{i}^{*}(t)$, and set
\begin{equation}
\bar{y}=\big((\gamma_{1}(t),\tilde{u}_{1}(t)),\ldots,(\gamma_{M}(t),\tilde{u}_{M}(t))\big),~t\in[0,t_f].
\end{equation}
\hfill $\Box$
\end{alg}
\begin{theorem}
Suppose that Assumption 2 and Assumption 3 are in force.
For every $\eta\in(0,1)$ there exists $\bar{\alpha}\in(0,1)$ such that, for every choices of
$\alpha\in(0,\bar{\alpha})$ and  $\beta\in(0,1)$,
Algorithm 2 has the property of sufficient descent with respect to the relaxed maximum principle.\hfill$\Box$
\end{theorem}
\begin{proof}
If Step 4 were to be replaced by $T(\bar{w}):=(1-\lambda_{\bar{w}})\bar{w}\oplus\lambda_{\bar{w}}\bar{u}_{\eta}^*$ then Algorithm 2
would be reduced to Algorithm 1, and by Theorem 1 it would have the sufficient-descent property. Therefore the theorem will be proved
once it is shown that, for $\bar{y}$ as defined
by Eq. (24) in Step 4, the following inequality is in force:
\begin{equation}
J(\bar{y})\leq J\big((1-\lambda_{\bar{w}})\bar{w}\oplus\lambda_{\bar{w}}\bar{u}_{\eta}^*\big).
\end{equation}
That is what we next prove.

{\it To simplify the notation in the  following discussion we omit the explicit dependence of various quantities on time $t$}.

By Eq. (24) and the definition of $\gamma_i(t)$ in Step 4,  it is seen that  $\bar{y}$ is an embedded control. In contrast,
 $(1-\lambda_{\bar{w}})\bar{w}\oplus\lambda_{\bar{w}}\bar{u}_{\eta}^*$ is
a relaxed control but not necessarily an embedded control. Nonetheless, we next prove that $\bar{y}$ and
 $(1-\lambda_{\bar{w}})\bar{w}\oplus\lambda_{\bar{w}}\bar{u}_{\eta}^*$ have the same associated state trajectories. Let $\{x(t)\}_{t\in[0,t_f]}$ denote the state trajectory
of $\bar{y}$. By Eqs. (24) and (17),
$\dot{x}=\sum_{i=1}^{M}\gamma_if_i(x,\tilde{u}_{i})$, and by the definition of $\tilde{u}_{i}$ (Step 4),
\begin{equation}
\dot{x}=\sum_{i=1}^{M}\gamma_{i}f_{i}(x,(1-\epsilon_{i})u_{i}+\epsilon_{i}u_{i}^*).
\end{equation}
By Eq. (20) and a bit of algebra,
\begin{equation}
\dot{x}=\sum_{i=1}^{M}\gamma_i\Big((1-\epsilon_{i})f_{i}(x,u_{i})+\epsilon_{i}f_{i}(x,u_{i}^*)\Big).
\end{equation}
Therefore, and by the definitions of $\epsilon_{i}$ and $\gamma_{i}$ in Step 4,
\begin{equation}
\dot{x}=\sum_{i=1}^{M}\Big((1-\lambda_{\bar{w}})\alpha_{i}f_{i}(x,u_{i})+\lambda_{\bar{w}}\alpha_{i}^*f_{i}(x,u_{i}^*)\Big).
\end{equation}
By Eq. (21) this is the state equation of $(1-\lambda_{\bar{w}})\bar{w}\oplus\lambda_{\bar{w}}\bar{u}_{\eta}^*$. Since both start at the initial
condition $x_{0}$, the two state trajectories are identical.

Next, consider the cost functions $J(\bar{y})$ vs.
$J\big((1-\lambda_{\bar{w}})\bar{w}\oplus\lambda_{\bar{w}}\bar{u}_{\eta}^*\big)$. By Eqs. (24) and (18),
\begin{equation}
J(\bar{y})=\sum_{i=1}^{M}\int_{0}^{t_f}\gamma_{i}L_{i}(x,\tilde{u}_{i})dt+\phi(x(t_f)).
\end{equation}
By the definition of $\tilde{u}_i$ in Step 4,
$L_i(x,\tilde{u}_i)=L_i(x,(1-\epsilon_i)u_i+\epsilon_iu_i^{*})$.
Therefore, and by Assumption 2(ii), $L_i(x,\tilde{u}_{i})\leq(1-\epsilon_{i})L_{i}(x,u_{i})+\epsilon_{i}L_{i}(x,u_{i}^{*})$.
Plug this inequality in (29). By the definitions of $\epsilon_{i}$ and $\gamma_{i}$ in Step 4,
\begin{eqnarray}
J(\bar{y})\leq\nonumber \\
\sum_{i=1}^{M}
\int_{0}^{t_{f}}\big((1-\lambda_{\bar{w}})\alpha_{i}L_{i}(x,u_{i})
+\lambda_{\bar{w}}\alpha_{i}^*L_{i}(x,u_{i}^*)\big)dt+\phi(x(t_{f}).
\end{eqnarray}
Since the state trajectories of $\bar{y}$ and $(1-\lambda_{\bar{w}})\bar{w}\oplus\lambda_{\bar{w}}\bar{u}_{\eta}^*$
are identical, and by Eq. (22), we recognize the RHS of (30) as
$J\big((1-\lambda_{\bar{w}})\bar{w}\oplus\lambda_{\bar{w}}\bar{u}_{\eta}^*\big)$.
This establishes that
$J(\bar{y})\leq J\big((1-\lambda_{\bar{w}})\bar{w}\oplus\lambda_{\bar{w}}\bar{u}_{\eta}^*\big)$, which completes the proof.
\end{proof}

\section{Examples}

This section presents three examples: an autonomous switched-mode system, an unstable hybrid system, and a spring-mass
 damper system.
The algorithm was coded by a MATLAB script, and run on a system based on an Intel Core i5 processor
 with 2.8 GHz clock.  All of the numerical integrations were performed by the forward Euler method or the
trapezoidal method.

\subsection{Curve tracking in a double-tank system}
Consider two cylindrical fluid tanks situated one on top of the other, each having  a hole at the bottom.  Fluid flows into each tank from the top and out through the hole.
The input flow to the upper tank comes from a valve-controlled hose, and the input flow to the lower
tank consists of the output flow from the upper tank. Let $v(t)$ denote the input flow rate to the upper tank, and let $x_{1}(t)$ and $x_{2}(t)$ denote the
amount of fluid in the upper tank and lower tank, respectively. $v(t)$ is the control input to to the system, and $x(t):=(x_{1}(t),x_{2}(t))^{\top}$ is
its state variable. By Toricelli's law the state equation is
\begin{equation}\label{1}
\begin{aligned}
\dot{x}(t) = \begin{pmatrix}
v(t)-\sqrt{x_1(t)}\\
\sqrt{x_1(t)}-\sqrt{x_2(t)}
\end{pmatrix},
\end{aligned}
\end{equation}
and we assume that the initial state is
$x(0) = (2.0, 2.0)^{\top}$. The control input $v(t)$ is assumed to be constrained to the two-point set $V := \left \{1.0, 2.0\right \}$, and hence
we can view the system as having two modes, mode 1 when $v(t)=1$, and mode 2 when $v(t)=2$. Using the modal notation, we can write
the state equation  as $\dot{x}\in\{f_{1}(x),f_{2}(x)\}$ with $f_{i}(x)$ defined by the RHS of (31) with $v(t)=i$, $i=1,2$.
We consider the problem of having
 the fluid level in the lower tank track a reference
curve $\{r(t)\}_{t\in[0,t_f]}$ for a given $t_f>0$, and correspondingly we minimize the cost functional
\begin{equation}\label{2}
\begin{aligned}
J:=2\int_{0}^{t_f}(x_2(t)-r(t))^2dt.
\end{aligned}
\end{equation}
This is an autonomous switched-mode system without a continuous-valued control $u$. Therefore the Hamiltonian function is
$H(x,v,p)=p^{\top}f(x,v)+L(x,t)$,  with $L(x,t)=2(x-r(t))^2$, and for given $x\in R^2$ and $p\in R^2$, its pointwise minimizer is $v^{*}(t)\in\{1,2\}$.
A measure $\mu\in M$ can be represented by a point $p\in[0,1]$, where $\mu(\{1\})=p$ and $\mu(\{2\})=1-p$, and hence a relaxed control
is a function $\mu:[0,t_f]\rightarrow[1,2]$.
Such systems are simpler than the controlled-systems discussed in Section III, the Hamiltonian is easily minimized (pointwise),
and Algorithm 2 is reduced to Algorithm 1. We provide this example nonetheless in order to highlight some features of the algorithm.

This problem was addressed  in \cite{Vasudevan13,Meyer12,Wardi16} with a constant target $r(t)=3.0$, while here we track  the  time-varying target curve $r(t)=0.5sin(0.1\pi t)+2.5$ over $t\in[0,30]$.
The algorithm's parameters are $\alpha=0.5$ and $\beta=0.5$.  All of the numerical integrations are performed
by the forward Euler method with the time step $\Delta t=0.01$. The initial control is $v(t)=2$ $\forall~t\in[0,t_f]$, and its cost
is $J(\bar{v}_{1})=84.185$.

The algorithm was run for 100 iterations, and its execution  took 17.207 seconds of cpu time. Figure 1
depicts the graph of $J(\bar{v}_{k})$ vs. the iteration count $k$,
and it exhibits sharp decrease before flattening after about 10  iterations. The final cost is $J(\bar{v}_{100})=2.627$, and the
graphs of the corresponding $x_{2}(t)$ (solid curve) and its target $r(t)$ (dashed curve) are
shown in Figure 2 for the sake of comparison.

\vspace{.2in}
\begin{figure}[h]
\centering
\includegraphics[width=0.35\textwidth]{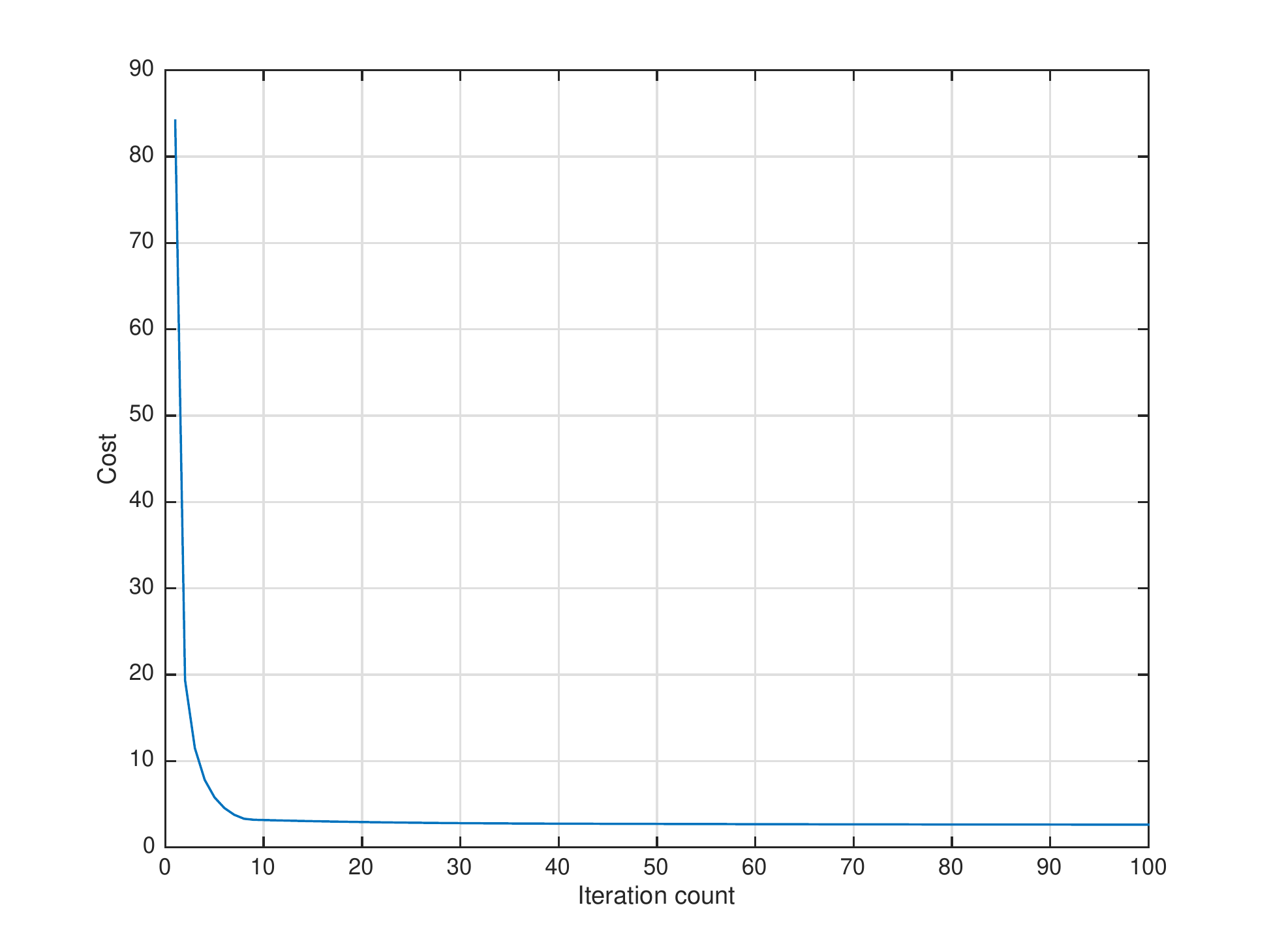}
{\small \caption{Double-tank system: Cost function vs. $k$}}
\end{figure}

\vspace{.2in}
\begin{figure}[h]
\centering
\includegraphics[width=0.35\textwidth]{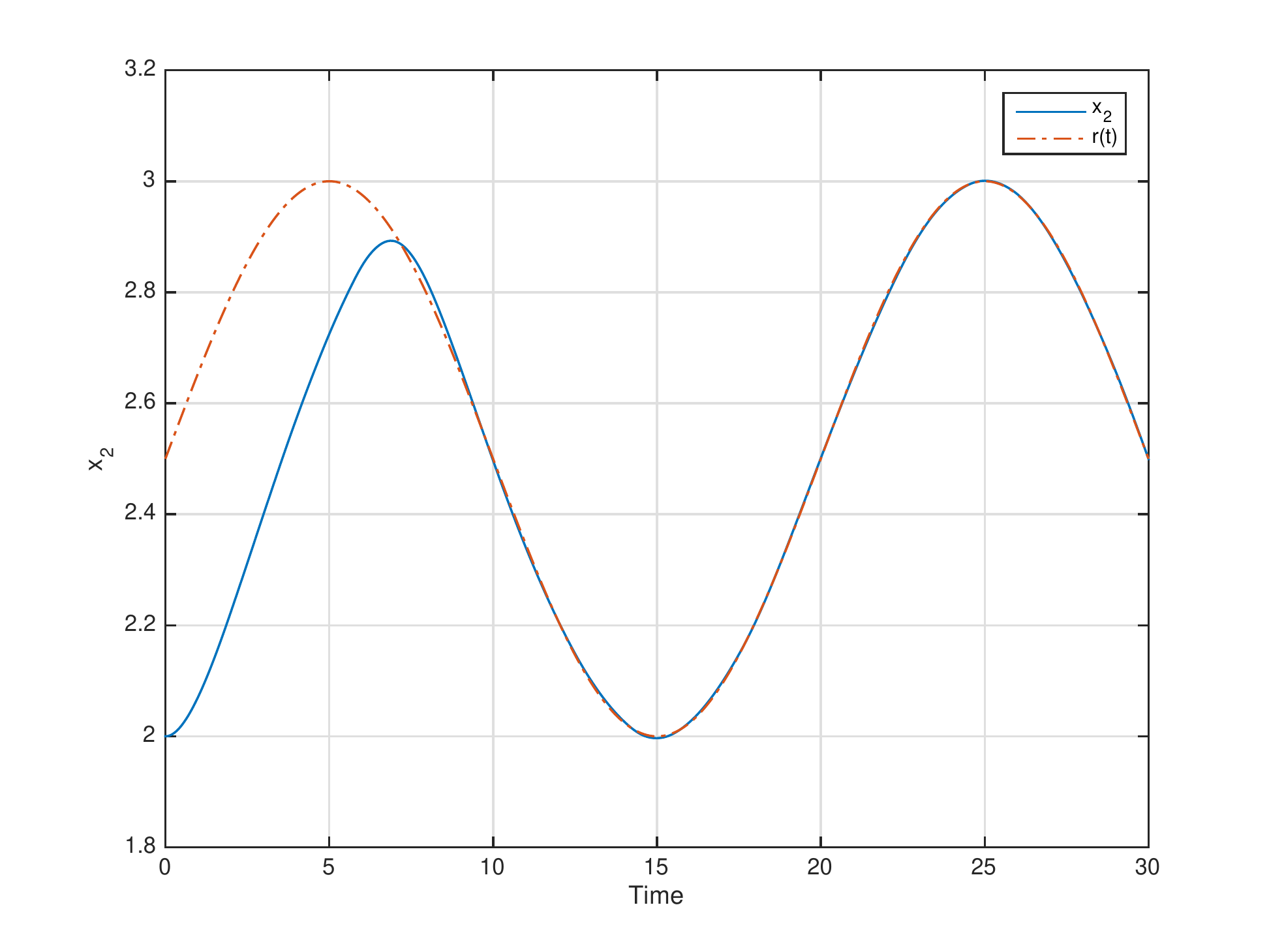}
{\small \caption{Double-tank system: $x_{2}(t)$ and $r(t)$}}
\end{figure}


Figure 1 indicates  a fast approach of the cost $J(\bar{v}_{k})$ towards its minimum value. Such $L$-shaped graph is not uncommon for descent algorithms
with Armijo step size. Its asymptotic convergence typically is slower than that of algorithms with superlinear convergence rate  \cite{Polak97}. On the other hand
they are descent methods that often take large strides towards local-minimum points at the early phases of their runs, hence the L-shaped graph in
Figure 1. After the run we projected the final relaxed control $\bar{v}_{100}$ onto the space of ordinary controls by
using pulse-width modulation, and the resulting control, denoted by $\bar{v}_{fin}$, has a cost of $J(\bar{v}_{fin})=2.7051$.

Finally, in order to explore ways to reduce the run times of the algorithm we experimented with fewer iterations and larger
integration step sizes. The results are summarized in Table 1, where $\Delta t$ indicates the integration step, $k$ is the number of iterations, $J(\bar{v}_{1})$ indicates the initial cost, $J(\bar{v}_{k})$ is the final cost of the algorithm, and $CPU$ is the cpu time of the run in seconds. We note that an increase of $\Delta t$ by a factor of 10 results in cpu reduction by about  a factor of 10
but with little increase in the final cost.

\begin{table}[h]
\centering
 \begin{tabular}{||c | c |c | c | c||}
 \hline
 $\Delta$ & k & $J(\bar{v}_{1})$ & $J(\bar{v}_{k})$ & CPU \\
 \hline\hline
 0.01 & 100 & 84.185 & 2.627 & 17.207 \\
 \hline
  0.01 & 50 & 84.185 & 2.7482 & 7.9468 \\
 \hline
 0.1 & 100 & 84.883 & 2.662 & 1.433 \\
 \hline
  0.1 & 50 & 84.883 & 2.7382 & 0.8138 \\
  \hline
\end{tabular}
\caption{Double tank problem}
\end{table}

\subsection{Control of an unstable hybrid system}
The following LQR system was considered in \cite{Xu04,Shaikh07}. The system has two modes, indexed by $i=1,2$. The dynamic
response functions  are
 $f_{i}(x,u)=A_{i}x+b_{i}u$,
where $x\in R^2$, $u\in R$, $A_{i}\in R^{2\times 2}$, and $b_{i}\in R^{2\times 1}$. The matrices
$A_{i}$ and $b_{i}$ are
\[
A_{1}=\left(
\begin{array}{cc}
0.6 & 1.2\\
-0.8 & 3.4
\end{array}
\right),\ \ \ \ \ A_{2}=\left(
\begin{array}{cc}
4.0 & 3.0\\
-1.0 & 0
\end{array}
\right),
\]
$b_{1}=(1,1)^{\top}$, and $b_{2}=(2,-1)^{\top}$. The initial condition is $x_0=(0,2)^{\top}$, and the final time
is $t_f=2.0$. The cost functional is
$
J=\int_{0}^{2}\frac{1}{2}\big(x_{2}(t)-2)^2+u(t)^2\big)dt
+\frac{1}{2}\big(x_{1}(2)-4\big)^2+\frac{1}{2}\big(x_2(2)-2\big)^2.
$
According to the problem formulation in Refs. \cite{Xu04,Shaikh07} the sequence of modes is fixed at $\{1,2\}$, and the control variable $v$ consists of the sole switching time between them and the continuous-valued input $u(t)$, $t\in[0,2]$.
In this paper the control variable consists of the mode-schedule without restrictions, and the continuous-valued control $u(t)$,
$t\in[0,2]$.
We use the trapezoidal method for integrating the differential equations. The initial guess
for the algorithm consists of mode 1 and $u(t)=0$ for all $t\in[0,2]$, and we ran the algorithm for 400 iterations.

The dominant eigenvalue of both matrices $A_{1}$ and $A_{2}$ is 3.0, hence the system is highly unstable. Therefore the algorithm did not work well
with single-shooting integrations of the state equation, and yielded a final cost of about 14.2, which is
higher than that obtained in \cite{Xu04,Shaikh07}
with a more-restricted control (9.766). Consequently we used multi-shooting integrations
 in the following way.
With $N$ denoting the number of shootings,  we divided the time-interval $[0,2]$ into $N$ equal-lengths subintervals
with end-points $0<\tau_{1}<\ldots<\tau_{N-1}<2$, introduced the additional variables $z_{j}$, $j=1,\ldots,N-1$ as the initial
condition for the state equation during the subinterval beginning at $\tau_{j}$, and added to the cost
the penalty term $K\sum_{j=1}^{N-1}||x(\tau_{j}^{-})-z_{j}||^2$. The penalty constant $K$ was determined by
the formula $K=2.5(N-1)$, since we felt that a higher penalty constant was needed for larger numbers of shooting intervals.
The integration step size was set to $\Delta t=\frac{0.1}{N-1}$.

After some experimentation we chose $N=10$,  hence $K=22.5$ and $\Delta t=0.011$. A  400-iteration run of the algorithm took 14.4 seconds of cpu time, and yielded the final cost of
$J(\bar{v}_{400})=7.0913$. Additional runs supported this result and indicated  that the obtained cost is practically
close to the minimum.
Moreover, the graph of $J(\bar{v}_{k})$ vs. $k=1,2,\ldots$ displays a similar L-shaped curve as in Figure 2. The final state trajectories $x_{1}(1)$ and $x_{2}(t)$ are shown in Figure 3.

\vspace{.2in}
\begin{figure}[h]
\centering
\includegraphics[width=0.35\textwidth]{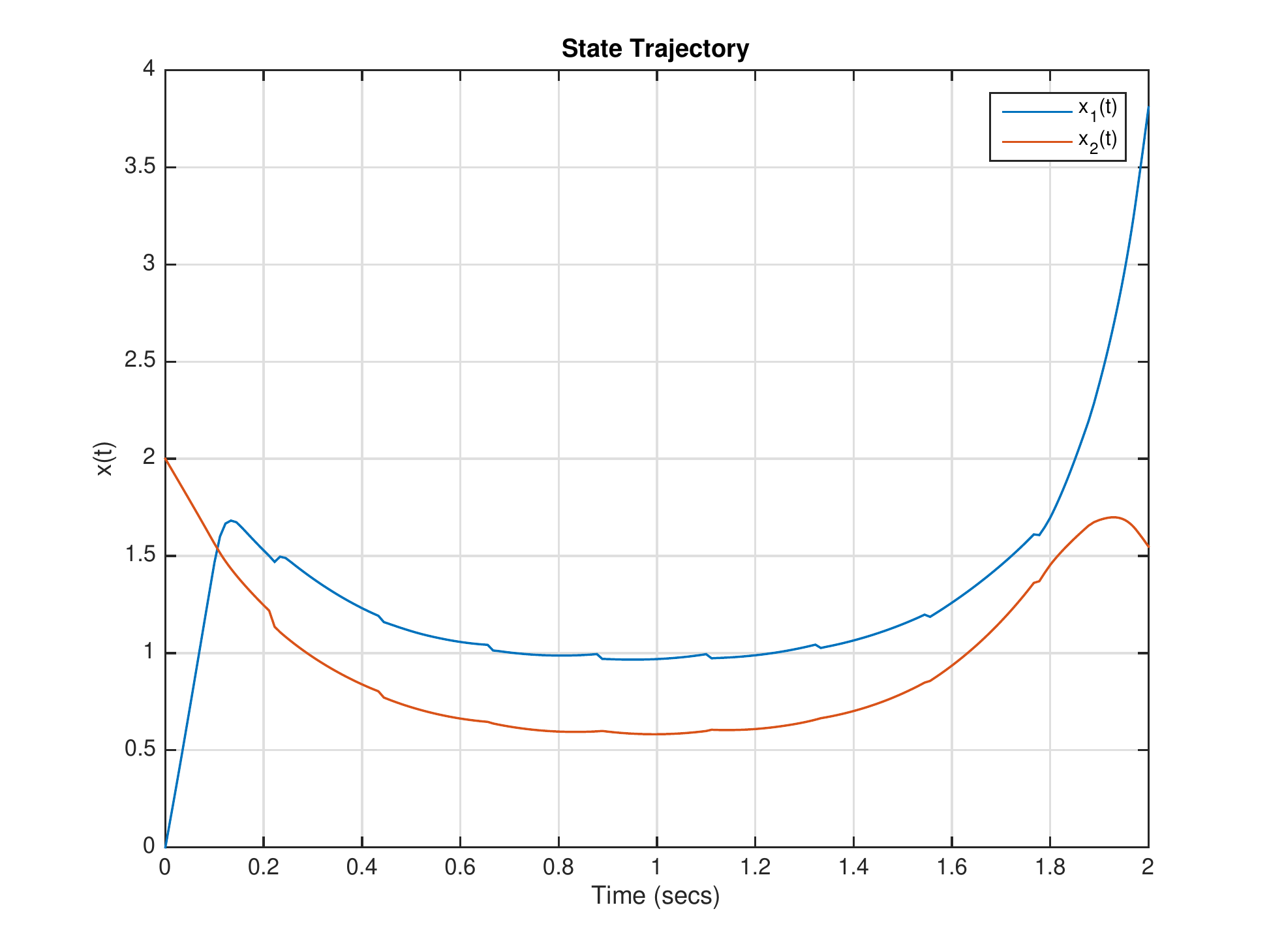}
{\small \caption{Unstable hybrid system: $x_{1}(t)$ and $x_{2}(t)$}}
\end{figure}

\subsection{Controlling a mass-spring damper system}

The problem described in this subsection has been considered in \cite{Borrelli13} which applied to it model-predictive control,
and a similar problem was solved in \cite{Bengea05} by a numerical algorithm.

Consider a mass connected to ground by
a  spring in series with a damper that represents viscous friction. Let $u(t)$ be an applied external force, and let $x_{1}(t)$ and
$x_{2}(t)$ be the mass' position and velocity. The system has two modes, indexed by $i\in\{1,2\}$, representing two levels of viscosity.
The state equation is
\begin{equation}
\begin{aligned}
\dot{x_{1}}(t) &= x_{2}(t) \\
M\dot{x_{2}}(t) &= -k(x_{1}(t)) - b_{i}x_{2}(t) + u(t),
\end{aligned}
\end{equation}
where $M$  the mass; the spring coefficient $k(x_{1})$ is $k(x_1)=x_1+1$ if $x_1\leq 1$, and $k(x_1)=3x_1+7.5$ if
$x_1>1$;
and the viscous friction coefficient is $b_1=1$ and $b_2=50$. The initial condition is $x_0=(3,4)^{\top}$.
We take the mass to be $M=1$. The cost functional is
$J=\int_{0}^{t_f}\big(||x(t)||^2+L_{i}(u(t))\big)dt+||x(t_f)||^2$,
where the mode-dependent cost function
is $L_1(u)=0.2u^2$, and $L_{2}(u)=0.2u^2+1$. We impose the constraints that, for all $t\in[0,t_f]$,  $|x_{j}(t)|\leq 5$, $j=1,2$,  and   $|u(t)|\leq 10$; and the finel-state constraint $|x_{j}(t_f)|\leq 0.01$, $j=1,2$.
We chose the final time to be $t_f=12.0$.

In order to satisfy the final-state constraints we appended the cost
functional by the penalty term $5||x_{1}(t_{f})||^2+30||x_{2}(t_{f})||^2$. We applied Algorithm 2 with the  parameters, $\alpha=0.01$ and  $\beta=0.5$, and the integration step size $\Delta t=0.01$.
The initial guess was $\alpha_{1}(t)=1$ (i.e., mode 1), and $u_{1}(t)=u_{2}(t)=0$ for all $t\in[0,t_f]$. 50 iterations took
11.6950 seconds of cpu time, and reduced the cost from its initial value of 94.0906 to its final value of 14.5166. The state
trajectory is shown in Figure 4, and the final state is $x(t_f)=(0.001,-0.0076)^{\top}$. A PWM-based projection of the final embedded
control
onto the space of ordinary controls incurs the cost $J$, excluding the penalty term, of $15.1954$.

\vspace{.2in}
\begin{figure}[h]
\centering
\includegraphics[width=0.35\textwidth]{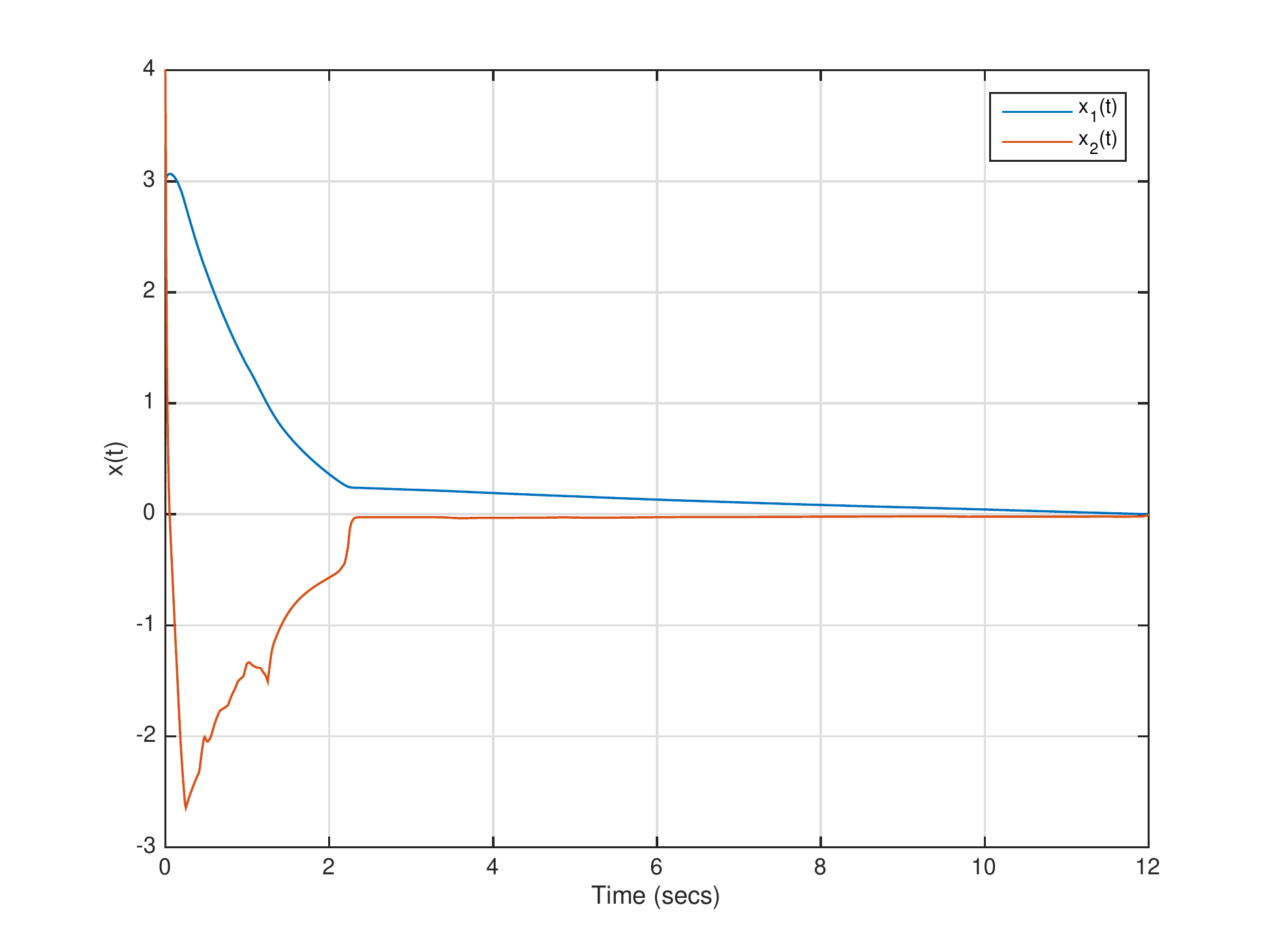}
{\small \caption{Mass-spring-damper system: $x_{1}(t)$ and $x_{2}(t)$}}
\end{figure}


\begin{thebibliography}{99}

\bibitem{Polak97}
E. Polak. {\it Optimization Algorithms and Consistent
Approximations}. Springer-Verlag, New York, New York, 1997.

\bibitem{Caldwell11}
T. Caldwell and T. D. Murphey. Switching mode generation and optimal
estimation with application to skid-steering, {\it Automatica}, vol. 47,
no. 1, pp. 50–64, 2011.

\bibitem{Caldwell12}
T. Caldwell and T. D. Murphey. Single integration optimization of
linear time-varying switched systems. {\it IEEE Transactions on Automatic
Control}, vol. 57, no. 6, pp. 1592–1597, 2012.

\bibitem{Miller13}
L. Miller and T.D. Murphey.
Simultaneous optimal estimation of mode transition times
and parameters applied to simple traction models.
{\it IEEE Transactions on Robotics}, vol. 29, no. 6, pp. 1496-1503, 2013.

\bibitem{Vasudevan13}
R. Vasudevan, H. Gonzalez, R. Bajcsy, and S.S. Sastry. Consistent
Approximations for the Optimal Control of Constrained Switched
Systems - Part 1: A Conceptual Algorithm, and Part 2: An Implementable Algorithm.  {\it SIAM Journal on Control and Optimization}, Vol. 51, pp.
4663-4483 (Part 1) and pp. 4484-4503 (Part 2), 2013.

\bibitem{Caldwell16}
T. Caldwell and T. Murphey.  Projection-Based Iterative Mode
Scheduling for Switched Systems, {\it Nonlinear Analysis: Hybrid Systems},
to appear, 2016.

\bibitem{Wardi16}
M.T. Hale, Y. Wardi, H. Jaleel, M. Egerstedt.
Hamiltonian-Based Algorithm for Optimal Control.
{\it Arxiv}, http://arxiv.org/abs/1603.02747. Also submitted to {\it Nonlinear Analysis:
Hybrid Systems}.

\bibitem{McShane67}
E.J. McShane.
Ralexed Controls and Variational Problems.
{\it SIAM Journal on Control},
Vol. 5, pp. 438-485, 1967.

\bibitem{Young69}
 L. C. Young.
{\it Lectures on the calculus of variations and optimal control theory}.
Foreword by Wendell H. Fleming. W. B. Saunders Co., Philadelphia, 1969.

\bibitem{Warga72}
J. Warga,
{\it Optimal Control of Differential and Functional Equations},
Academic Press, 1972.

\bibitem{Gamkrelidze78}
R. Gamkrelidze. {\it Principle of Optimal Control Theory.} Plenum, New York, 1978.

\bibitem{Vinter00}
R. Vinter.
{\it Optimal Control}, Birkhauser, Boston, Massachusetts, 2000.

\bibitem{Lou09}
H. Lou.
Analysis of the Optimal Relaxed Control to an Optimal Control Problem.
{\it Applied Mathematics and Optimization}, Vol. 59, pp. 75-97, 2009.

\bibitem{Berkovitz13}
L.D. Berkovitz and N.G. Medhin.
{\it Nonlinear Optimal Control Theory},
Chapman \& Hall, CRC Press,
Boca Raton, Florida, 2013.

\bibitem{Xu02}
X. Xu and P.J. Antsaklis. Optimal Control of Switched Systems via
Nonlinear Optimization Based on Direct Differentiations of Value
Functions. {\it International Journal of Control}, Vol. 75, pp.
1406-1426, 2002.

\bibitem{Shaikh02}
M.S. Shaikh and P. Caines. On Trajectory Optimization for Hybrid
Systems: Theory and Algorithms for Fixed Schedules. {\it IEEE
Conference on Decision and Control}, Las Vegas, NV, Dec. 2002.

\bibitem{Xu04}
X. Xu and P. Antsaklis.
Optimal control of switched-systems based on parameterization
of the switching instants.
{\it IEEE Trans. Automatic Control}, vol. 49, pp. 2-16, 2004.

\bibitem{Egerstedt06}
M. Egerstedt, Y. Wardi, and H. Axelsson. Transition-Time
Optimization for Switched Systems.  {\it IEEE Transactions on
Automatic Control}, Vol. AC-51, No. 1, pp. 110-115, 2006.

\bibitem{Shaikh03}
M.S. Shaikh and P.E. Caines. On the Optimal Control of Hybrid
Systems: Optimization of Trajectories, Switching Times and
Location Schedules. In {\it Proceedings of the 6th International
Workshop on Hybrid Systems: Computation and Control}, pp. 466-481,
Prague, The Czech Republic, 2003.

\bibitem{Caines05}
P. Caines and M.S. Shaikh. Optimality Zone Algorithms for Hybrid
Systems Computation and Control: Exponential to Linear Complexity.
{\it Proc. 13th Mediterranean Conference on Control and
Automation}, Limassol, Cyprus, pp. 1292-1297, June 27-29, 2005.

\bibitem{Shaikh05}
M.S. Shaikh and P.E. Caines.
Optimality Zone Algorithms for Hybrid Systems Computation and Control:
From Exponential to Linear Complexity.
{\it Proc. IEEE Conference on Decision and Control/European Control
Conference}, pp. 1403-1408, Seville, Spain, December 2005.

\bibitem{Shaikh07}
M.S. Shaikh and P.E. Caines.
On the Hybrid Optimal Control Problem: Theory and Algorithms.
{\it IEEE Trans. Automatic Control}, Vol. 52, pp. 1587-1603, 2007.


\bibitem{Hellund99}
S. Hedlund and A. Rantzer.
Optimal control for hybrid systems.
{\it Proc. 38th CDC},
Phoenix, Arizona, December 7-10, 1999.

\bibitem{Attia05}
S.A. Attia, M. Alamir, and C. Canudas de Wit. Sub Optimal Control of
Switched Nonlinear Systems Under Location and Switching Constraints.
{\it Proc. 16th IFAC World Congress}, Prague, the Czech Republic,
July 3-8, 2005.


\bibitem{Gonzalez10}
H. Gonzalez, R. Vasudevan, M. Kamgarpour, S.S. Sastry, R. Bajcsy, and C. Tomlin.
A Numerical Method for the Optimal Control of Switched Systems.
{\it Proc. 49th CDC}, Atlanta, Georgia, pp. 7519-7526, December 15-17, 2010.

\bibitem{Wardi12}
Y. Wardi and M. Egerstedt.
Algorithm for Optimal Mode Scheduling in  Switched Systems.
{\it Proceedings of the American Control Conference}, Montreal, Canada, June 2012.

\bibitem{Taringoo13}
F. Taringoo and P.E.  Caines.
On the Optimal Control of Impulsive Hybrid Systems on Riemannian Manifolds.
{\it SIAM Journal on Control and Optimization},  Vol. 51, Issue 4, pp. 3127 - 3153, 2013.


\bibitem{Ge75}
X. Ge, W. Kohn, A. Nerode, and J.B. Remmel.
Hybrid systems: Chattering approximation to
relaxed controls.
{\it Hybrid Systems III: Lecture Notes in Computer Science},
R. Alur, T. Henzinger, E. Sontag, eds., Springer Verlag,
Vol. 1066, pp. 76-100, 1996.

 \bibitem{Bengea05}
S.C. Bengea and R. A. DeCarlo.
Optimal control of switching systems.
{\it Automatica}, Vol. 41, pp. 11-27, 2005.


\bibitem{Meyer12}
R.T. Meyer, M. Zefran, and R.A. Decarlo. Comparison of the Embedding Method to
Multi-Parametric Programming, Mixed-Integer Programming, Gradient Descent, and
Hybrid Minimum Principle Based Methods.
{\it IEEE Transactions on Control Systems
Technology},
Vol. 22, no. 5,  pp. 1784-1800, 2014.




\bibitem{Lin14}
H. Lin and P. J. Antsaklis.
Hybrid Dynamical Systems:
An Introduction to Control and Verification.
{\it Foundations and Trends  in Systems and Control},
Vol. 1, no. 1, pp. 1-172, March 2014.

\bibitem{Borrelli13}
F. Borrelli, A. Bemporad, and M. Morari.
{\it  Predictive Control for Linear
and Hybrid Systems}. Cambridge University Press, Cambridge, 2013.

\end{thebibliography}
\end{document}